\documentclass[12pt]{article}
\usepackage{amsmath,amsthm,amssymb}
\usepackage{amsfonts}

\newcommand{\url}[1]{}

\newcommand{\Px}{ \mathbb{P} }

\newcommand{\Real}{\mathbb R}
\newcommand{\Natural}{\mathbb{N}}

\newcommand{\ii}{\mathrm{i}}
\newcommand{\sgn}{\operatorname{\mathrm{sgn}}}

\def\keywordname{{\bfseries Key words:}}
\def\keywords#1{\par\addvspace\baselineskip\noindent\keywordname\enspace
\ignorespaces#1}

\newtheorem{theorem}{Theorem}[section]
\newtheorem{proposition}[theorem]{Proposition}

\newtheorem{definition}[theorem]{Definition}
\newtheorem{example}[theorem]{Example}
\newtheorem{remark}[theorem]{Remark}
\newtheorem{problem}[theorem]{Problem}

\title{\vspace{-2.5cm} {\bf Consistent single- and multi-step sampling of multivariate arrival times:\\ A characterisation of self-chaining copulas}}
\author{Damiano Brigo\thanks{Corresponding author. Email: {\tt firstname.familyname@kcl.ac.uk}; we are grateful to Jan-Frederik Mai for helpful correspondence that helped us improving the paper. We are also grateful to Alexander McNeil for a helpful discussion and references on Extreme Value Copulas.}  \ \ \ \ \ \ \  Kyriakos Chourdakis\\
Dept. of Mathematics, King's College, London 
}

\date{First version: November 20, 2007. This version: \today}

\begin{document}

\maketitle

\vspace{-1cm}

\begin{abstract} 
\noindent This paper deals with dependence across marginally exponentially distributed arrival times, such as default times in financial modeling or inter-failure times in reliability theory. We explore the relationship between dependence and the possibility to sample final multivariate survival in a long time-interval as a sequence of iterations of local multivariate survivals along a partition of the total time interval. We find that this is possible under a form of multivariate lack of memory that is linked to a property of the survival times copula. This property defines a ``self-chaining-copula", and we show that this coincides with the extreme value copulas characterization.  The self-chaining condition is satisfied by the Gumbel-Hougaard copula, a full characterization of self chaining copulas in the Archimedean family, and by the Marshall-Olkin copula.  The result has important practical implications for consistent single-step and multi-step simulation of multivariate arrival times in a way that does not destroy dependency through iterations, as happens when inconsistently iterating a Gaussian copula. 
\end{abstract}

\noindent {\bf AMS Classification Codes:} 60E07,
62H05, 62H20, 62H99



%
%
%
%
%

\thispagestyle{empty}

\tableofcontents

\newpage

\section{Introduction and Motivation}\label{introsec}
This paper deals with dependence across marginally-exponential arrival times, for example default times in financial modeling or inter-failure times in reliability theory. Dependence can be modeled both through a multivariate distribution function consistent with the exponential marginal distributions, or by joining the given exponential marginal distributions through a copula function. After introducing a multivariate lack of memory property, we first present two multivariate exponential distributions that are known in the literature and satisfy such property, and then move to copula functions that are consistent with these distributions and with the lack of memory property in particular. The lack of memory property is important because it allows for consistency between a unique survival sampling at a single large time interval and iterated survival samplings at smaller sub-intervals forming a partition of the whole interval. Since in the industry it is important to have the possibility to decompose a final survival simulation into survival subsimulations, this is also a relevant practical problem. The relevance stems from the need to make the survival simulation time step consistent with the common time step for all other underlying variables or risk factors (e.g. equity prices, interest rates, temperatures, etc). We provide a characterization of "self-chaining" copula functions satisfying the lack of memory property. We consider as fundamental examples the Marshall-Olkin copula and the Gumbel-Hougaard Copula. In the family of Archimedean copulas, we provide a full characterization, showing that the only lack-of-memory and consistent single-step and multi-step Archimedean copula is the Gumbel-Hougaard copula. 

The paper is organized as follows. Section \ref{subsec:CPLM} introduces the common period lack of memory (CPLM) property that makes consistency between single-step and multi-step sampling possible. The following Section \ref{sec:introMVLM}  introduces multivariate exponential distributions that satisfy such lack of memory property. This is done directly at multivariate distribution level without splitting marginal and dependence information, i.e. without resorting to copula functions. We resort to copula functions in the following,  so that Section \ref{subsec:copfun} presents a quick introduction to copula functions, whereas  Section \ref{subsec:scc} introduces the self-chaining property that translates the lack of memory and consistency above into copula language.  Section \ref{sec:copulaarrival} goes into more detail and derives the characterization of the common period lack of memory property in terms of copula properties, proving that this is achieved by the self-chaining property. A first characterisation of self-chaining copulas in terms of homogeneity or PDEs is presented.  
Section \ref{sec:archchar} fully characterises self chaining copulas in the archimedean family as the Gumbel-Hougaard copula.
Section \ref{sec:sccevc} explains that self-chaining and extreme-value copulas are the same, and by summarizing the Pickands function characterization of Extreme Value Copulas provides a precise characterization also of self-chaining copulas. Finally, Section \ref{sec:conclu} concludes the paper.

\section{A multivariate form of lack of memory: CPLM}\label{subsec:CPLM}
Consider event times (typically inter-arrival times, inter-failure times or for example defaults) for $n$ entities, and let the event time for entity $i$ be denoted by $\tau_i$. These event times can for example be default times of a number of possibly related firms or first-failure times for a number of entities.

Denote by $F_i$ the cumulative distribution function for $\tau_i$, and by $G_i = 1-F_i$ the related survival function. 
Suppose now that we look at a single period vs multi-period context. Consider a final period $N T$ obtained by adding a partition of single period intervals of length $T$, hence periods $[0,T), [T,2T), \ldots, [(N-1)T,NT]$.

Assume that the single event times do not have memory, in that
\begin{equation}\label{eq:nomemory} \Px(\tau_i > T | \tau_i > S) = \Px(\tau_i > T-S)
\end{equation}
for any $i$ and any $0\le S \le T$ (and hence are exponentially distributed and $G_i(N T) = G_i(T)^N$).
An important practical question is the following. Is it possible to iterate a simulation of the joint survival of arrival times always in the same way in all subintervals 
$[0,T), [T,2T), \ldots, [(N-1)T,NT]$, and also in the same way as we simulate joint survival of arrival times in a single sampling run in $[0, N T]$? The answer is affirmative under a multivariate form of 
{\em lack of memory}, namely  (for the case $n=2$ for simplicity)
\begin{equation}\label{eq:lombiv}  \Px(\tau_1 \ge j T,\tau_2 \ge j T|\tau_1 \ge h T,\tau_2 \ge h T) = \Px(\tau_1 \ge (j-h)T,\tau_2 \ge (j-h)T) 
\end{equation}
with $j>h$ integers.
In this paper we will try and characterise this lack of memory property, which we term "common periods lack of memory" (CPLM), in accordance with the following remark. 

\begin{remark} {\bf (Lack of memory at common levels for both random times)} 
It should be noted, importantly, that we do not adopt the most general definition of bivariate lack of memory, namely
\begin{equation}\label{eq:lombivgen}  \Px(\tau_1 \ge j T,\tau_2 \ge k T|\tau_1 \ge h T,\tau_2 \ge i T) = \Px(\tau_1 \ge (j-h)T,\tau_2 \ge (k-i)T) 
\end{equation}
for $\max(i,h) \le \min(j,k)$.  
This is because we plan to apply our result to joint simulations of the multivariate vector of arrival times in common intervals. Besides, Condition in Eq. (\ref{eq:lombivgen}) would be too strong, implying a trivial case of lack of memory, namely independence of the two exponential random variables $\tau_1$ and $\tau_2$, see for example \cite{MarshallOlkin}. 
\end{remark}

\section{Multivariate exponential distributions and CPLM}\label{sec:introMVLM}
We now present two possible multivariate exponential distributions with exponential marginals and satisfying the Common Period Lack of Memory (CPLM) property. 

Consider the bivariate case for simplicity, involving a bivariate vector of univariate exponential arrival times  $\tau_1,\tau_2$.  Leaving aside the general metod of introducing a copula to connect marginally exponential random variables for the time being, it is well known that there are several bivariate exponential distributions that could be used to model the bivariate random vectors consistently with the univariate exponential distributions. One of the most utilized bivariate exponential distributions is the Marshall-Olkin bivariate distribution (see \cite{MarshallOlkin}). This distribution generalizes the lack of memory property of univariate exponential distributions, satisfying in particular our CPLM definition in Eq (\ref{eq:lombiv}) above.
The Marshall Olkin bivariate distribution, however, features a singular component and admits a strictly positive probability that $\tau_1 = \tau_2$. When one excludes perfectly simultaneous $\tau_1$ and $\tau_2$, one may resort  for example to a bivariate exponential distribution among the three proposed by Gumbel in \cite{Gumbel}. The first bivariate distributions in \cite{Gumbel} satisfies an alternative definition of lack of memory, also known as bivariate remaining life constancy, see \cite{Kotz}.
However, this first distribution proposed by Gumbel can only describe negative dependence and in a limited range. The second bivariate exponential in \cite{Gumbel} only describes a range of dependence $[-1/4, 1/4]$ for correlation, and as such is of limited scope. This is why we resort to the third bivariate exponential distribution only briefly introduced in \cite{Gumbel}. See also \cite{LuBhatta} and \cite{Kotz}.

The joint survival function for this bivariate exponential is, for simple univariate exponential marginals with positive intensities $\lambda$'s, and for dependence parameter $\theta  \in [1, \infty)$,
\begin{equation}\label{eq:Gumbexp} \mathbb{Q}(\tau_1 > x_1,\tau_2 > x_2) := G(x_1,x_2) = \exp( - ( (\lambda_1 x_1 )^\theta + (\lambda_2 x_2 )^\theta )^{1/\theta} ). 
\end{equation}

Notice that, indeed, the marginal distributions are exponential random variables with mean respectively $1/\lambda_1$ and $1/\lambda_2$. We will set $\lambda_1$ to a constant deterministic intensity, and $\lambda_2$ to a constant intensity as well. It is straightforward to check that this third Gumbel multivariate exponential distribution satisfies the lack of memory property in Eq (\ref{eq:lombiv}) (i.e. CPLM, but not the more general lack of memory property in Eq \ref{eq:lombivgen}!).

We notice that Kendall's tau for this distribution, which is a good measure of dependence (invariant for invertible increasing transformations), is
\begin{equation}
 \tau^K(G) = 1 - 1/\theta.
\end{equation}
This confirms that $\theta=1$ characterizes the independence case, whereas $\theta \rightarrow \infty$ characterizes the co-monotonic case. This bivariate exponential distribution can therefore describe the whole range of positive dependence.

We can also notice that $\lambda$'s are pure marginal parameters, whereas $\theta$ is a pure dependence parameter. This is a bivariate distribution allowing for tail dependence, and, differently from Marshall Olkin's, does not have a singular component, so that there is zero probability that $\tau_1 = \tau_2$.

It should be noted that in a number of applications, and especially in financial modeling of default times, the assumption that $\lambda$'s are deterministic is not realistic and its negative features have been highlighted for example in \cite{Brigo08}. 

The Marshall-Olkin and Gumbel multivariate exponential distributions satisfy the CPLM property and thus lend themselves to a consistent single-step and multi-step simulation and sampling of survival. However, to better characterize dependency that is consistent with multivariate lack of memory we standardize away marginal distributions and focus on copula functions as follows.   

\section{Copula Functions}\label{subsec:copfun}
Consider a random vector $X=(X_1,...,X_n)$, and suppose that we
wish to analyze the dependence between its components. The law of
$X$ is then characterized by its joint cumulative distribution
function (CDF) $(x_1,\ldots,x_n) \mapsto \Px(X_1\le x_1,...,X_n\le
x_n)$. However, this function mixes information on the dependence
between the different components of the vector with information on
the distribution of the single components themselves. Copula
functions have been introduced in order to allow for a separation
between the marginal CDF and the dependence structure. The former
concerns single components, taken one at the time, and is given by
the CDF's $F_i(x):=\Px(X_i\le x)$, $i=1,\ldots,n$, which we assume
to be continuous. The latter is entirely represented by the copula
function we introduce now. It is well known that
$U_1=F_1(X_1),...,U_n=F_n(X_n)$ are uniformly distributed random
variables on $[0,1]$. The joint cumulative distribution function
of $(U_1,...,U_n)$, that we denote by $C(u_1,...,u_n)= \Px(U_1 \le
u_1,...,U_n \le u_n)$,  is called the copula function of
$(X_1,...,X_n)$ and satisfies:
\begin{eqnarray} \label{cop1}
   \Px(X_1\le x_1,...,X_n\le x_n)=C(\Px(X_1\le x_1),...,\Px(X_n\le
   x_n)).
\end{eqnarray}
A copula function has the following three properties:
\begin{eqnarray}
\label{cop:1} C(u_1,..,u_{i-1},0,u_{i+1},..,u_n)=0 \label{eq:coppro1} \\
\label{cop:2} C(1,..,1,u_k,1,..,1)=u_k \label{eq:coppro2}  \\
\label{cop:3} \partial_{u_1...u_n}C\  \mbox{is a
positive measure in the sense of Schwartz distributions.} \label{eq:coppro3} 
\end{eqnarray}
Condition (\ref{cop:3})
means concretely that for any hypercube\\ 
$H=[a_1,b_1]\times...\times[a_n,b_n] \subset [0,1]^n$ we have
$\Bbb{P}[(U_1,..,U_n) \in H]\geq 0.$

Conversely, one can show that any function that satisfies the above
three conditions (\ref{eq:coppro1}, \ref{eq:coppro2}, \ref{eq:coppro3})  can be viewed as the joint CDF of a vector of
uniform variables on $[0,1]$ and is thus a copula, see for example \cite{Joe}, \cite{Nel}, \cite{ressel} and \cite{sklar}.  
In this spirit, {\it in the following, the expression "simulating
a copula C" will denote the simulation of a random vector of
uniform variables $(U_1,..,U_n)$ on $[0,1]$ whose joint CDF is
$C$}. This can be occasionally referred to as ''sampling the copula".

Among the different ways to define specific copula functions,
there are the following two. The first one consists in seeking
functions $C$ satisfying the three above properties. Archimedean
copulas are an example of this approach. Indeed, Archimedean
copulas stem from the remark that if $\varphi$ is a convex
strictly decreasing function such that $\varphi(1)=0$ and with suitable additional properties (see Theorem \ref{th:kimberling} below), then
$C(u_1,..,u_n)= \varphi^{[-1]} (\varphi(u_1)+..+\varphi(u_n))$
has the above three properties and is thus a copula. The function $\varphi$
is called the generator of the copula.  The Archimedean copula is said to be strict if the generator satisfies
$\lim_{u \rightarrow 0^+ } \varphi(u) = \infty$, and for these copulas the pseudo-inverse $\varphi^{[-1]}$ can be replaced by the inverse $\varphi^{-1}$ and the generator is said itself to be strict.
Examples of Archimedean copulas are the Gumbel-Hougaard, Joe, Clayton, Frank copulas, see Nelsen (1999) and Joe (1997).

The second method consists in
working directly with joint CDF's $F(x_1,...,x_n)$ and the related
marginal CDF's $F_i$, providing the copula
$F(F_1^{-1}(u_1),...,F_n^{-1}(u_n))$. Even if this method does not
always lead to analytically tractable copulas, it can provide us
with copulas that are easy to simulate. Indeed, the main example
of this kind of construction is the fundamental family of Gaussian
copulas. Student-t copulas are also defined this way. 

A Gaussian copula is defined as the copula of a joint
Gaussian random vector $X$ with standard Gaussian marginals and
correlation matrix $\rho$, and is thus given by
$N_{\rho}(N^{-1}(u_1),...,N^{-1}(u_n))$ where $N$ is the CDF of a
scalar standard Gaussian variable and $N_{\rho}$ is the joint CDF
of $X$.
%
%

In recent years, copula functions have received a great deal of
attention, see for example the papers of Genz and Bretz (2002),
H\"{u}rlimann (2003), Juri and W\"{u}thrich (2002), Wei and 
Hu~(2002), Alfonsi and Brigo (2005), and the books of Joe~(1997) and Nelsen~(1999). For
financial and insurance applications, recent applications on
copulas include for example Li (2000), Bouy\'e et al. (2000), Cherubini et
al. (2002), Jouanin et al. (2001),
Klugman and Parsa (1999).
For an excellent extensive overview see Embrechts et al. (2003), and for applications to finance see \cite{cherubook}.

\section{Multivariate Sampling: Single vs Multi step copulas}\label{subsec:scc}

Consider the joint survival or no-event (no-failure) probability function at $t_1,\ldots,t_n$,
\[ \Px(\tau_1 \ge t_1,\ldots,\tau_n \ge t_n) = \Px(G_1(\tau_1) \le G_1(t_1),\ldots,G_n(\tau_n) \le G_n(t_n)) \]\[=
\Px(U_1 \le G_1(t_1),\ldots, U_n \le G_n(t_n)) =: C(u_1,\ldots,u_n)\]
where we termed $C$ the survival times copula and we set $U_i = G_i(\tau_i)$, uniform random variable, and $u_i = G_i(t_i)$.

Suppose now that we look again at a single period vs multi-period context. We still enforce univariate lack of memory as in 
Eq (\ref{eq:nomemory}).   
In this paper we will try and characterise the multivariate lack of memory property CPLM of Eq (\ref{eq:lombiv}) in terms of copula properties. 
We will see that if CPLM holds, we can sample the event where all arrival times exceed time $N T$ either at the final time $N T$ directly, which is consistent with a survival times copula
\[ C(u_1,\ldots,u_n), \ \ u_1 = G_1(N T), \ldots,  u_n = G_n(N T) , \]
or by iterating $N$ steps of survival in each $[(i-1)T, i T), \ i=1,\ldots,n$. We will see that under CPLM the second procedure produces the copula
\[ (C(u_1^{1/N},\ldots,u_n^{1/N}))^N, \ \ u_1 = G_1(N T), \ldots,  u_n = G_n(N T) . \]
Clearly, to have consistency the two copulas have to coincide, namely
\[ (C(u_1^{1/N},\ldots,u_n^{1/N}))^N = C(u_1,\ldots,u_n) \]
When this happens for a given copula function, we call the copula a "self-chaining copula", and this is consistent with common period multivariate lack of memory.

Now the problem is that the industry wrongly applies the above procedure or related procedures even when the copula is not self-chaining and when CPLM does not hold. Typically, the above procedure is applied with a Gaussian copula $C$. This is not consistent and leads to a difference between the one-shot sampling at $N T$ and the iterated sampling, to the point that the latter kills dependence for large $N$.

This is such an important point that it is worth clarifying it with an example.

\begin{example}\label{ex:iter}{\bf (Iterating a Gaussian copula kills dependence)}
Consider two exponentially distributed default times connected by a Gaussian copula with dependence parameter $\rho$ and with constant intensities $\lambda_1$ and $\lambda_2$, namely 
\[ \tau_1 = -\ln(1-\Phi(X_1))/\lambda_1, \ \ \tau_2 = -\ln(1-\Phi(X_2))/\lambda_2,\] 
where $[X_1,X_2]$ is a bivariate Gaussian random vector with standard gaussian marginals and with correlation parameter $\rho$. 

Assume $\lambda_1 = \lambda_2 = 0.02$, \ \ $\rho = 0.9$, and consider the following two procedures:

\begin{itemize}
\item Sample directly the event $\tau_1 > 100y \cap \tau_2 > 100y$.
The probability of this event, based on a simulation with one million scenarios, is

\[ 0.097 \pm 0.0003 \ \ (0.0969) \]
where we included the simulation standard error. The value within brackets is the value computed by resorting to a double numerical integration routine to compute the Gaussian copula.

\item Iterate the check $\Delta_i \tau_1 > 1y \cap \Delta_i \tau_2 > 1y$ 100 times, where 
$[\Delta_i \tau_1, \Delta_i \tau_2]$, $i=1,2,\ldots,100$ are independent copies of  $[\tau_1,\tau_2]$ to be used to check default in every year. We count the scenarios along which there is always joint survival up to 100 years and divide by the total number of scenarios. 

This yields a survival probability of 

\[ 0.057 \pm 0.0003  \ \ \  (0.0557)\]
The value within brackets is the value computed by resorting to a double numerical integration routine to compute the Gaussian copula for 1y, thus determining the 1y survival probability, and then using the 100th power of this number for the survival over 100 years.

\end{itemize}

As we can see, there is a quite relevant difference, in this case, between simulating joint survival one-shot for 100 years, and iterating 100 times a 1y survival event as if the multivariate arrival times satisfied lack of memory. In the latter case the joint survival probability (and hence dependence) is much smaller. This is a confirmation that iteration, when lack of memory is not satisfied, kills dependency. This happens, in particular, with the Gaussian copula.
\end{example}

%

In this paper we will characterize copula functions leading to CPLM and hence to the possibility to have consistent sampling and subsampling, or single-step and multi-step simulations which do not destroy dependence through iteration. This is of great practical relevance.

\section{Copula characterization of multivariate Lack of Memory}\label{sec:copulaarrival}
The main result of this section is the following
\begin{proposition}
Assume single event times $\tau_1$ and $\tau_2$ satisfy each the lack of memory property (\ref{eq:nomemory}).
{ Lack of Memory} extends to the bivariate default time via the CPLM condition in Eq (\ref{eq:lombiv})
if and only if the survival or no-event time copula satisfies
\begin{equation}\label{eq:copulacond} C(u_1^k,u_2^k) = C(u_1,u_2)^k  \ \  \ \ \mbox{ (self-chaining copula)}
\end{equation}
for all $u_1, u_2$ in $[0,1]$ and integer positive $k$.
\end{proposition}
\begin{proof}
Assume lack of memory holds in the bivariate case too. Compute
\begin{eqnarray*}  \Px(\tau_1 \ge 2 T,\tau_2 \ge 2 T) = \Px(\tau_1 \ge 2 T,\tau_2 \ge 2T|\tau_1 \ge T,\tau_2 \ge T)
\Px(\tau_1 \ge T,\tau_2 \ge T) =  \\
=\Px(\tau_1 \ge 2T-T,\tau_2 \ge 2T-T)\Px(\tau_1 \ge T,\tau_2 \ge T) = \Px(\tau_1 \ge T,\tau_2 \ge T)^2 .
\end{eqnarray*}
Now lack of memory for single default times implies easily $G_i(2T) = G_i(T)^2$, so that by taking $G$'s on both sides of the events whose probabilities are computed, from the previous equation we have
\begin{eqnarray*}  \Px(U_1 \le G_1(T)^2, U_2 \le G_2(T)^2) =  \Px(U_1 \le G_1(T), U_2 \le G_2(T))^2
\end{eqnarray*}
or actually (\ref{eq:copulacond}) with $k=2$. Iterating shows (\ref{eq:copulacond}) holds for all $k$.

Viceversa, if (\ref{eq:copulacond}) holds for all $k$, it is immediate to prove that lack of memory holds.
\end{proof}

\begin{remark} {\bf (Extension to real exponents)} \\ The change of variables $\bar{u}_{1,2} := u_{1,2}^k$ with a substitution in (\ref{eq:copulacond}) shows that the analogous of (\ref{eq:copulacond}) in $\bar{u}_{1,2}$ will hold with exponent $1/k$. Combining the two results for different $k$ yields easily that (\ref{eq:copulacond}) will hold for any rational positive exponent replacing the positive integer $k$, and by extending the relationship by continuity to the whole set of positive real exponents we have (\ref{eq:copulacond}) holding for any positive real $k$.
\end{remark}

The proof also extends straightforwardly to the copula on $n$ exponentially distributed arrival times.

We refer to Condition (\ref{eq:copulacond}), that in the $n$-dimensional case reads
\begin{equation}\label{eq:copulacondndim} C(u_1^k,u_2^k,\ldots,u_n^k) = C(u_1,u_2,\ldots,u_n)^k , k \in {\Bbb R}^+  \ \ \mbox{ (self-chaining copula)},
\end{equation}
as to a ``self-chaining" property of the copula. This property tells us that if we break the simulation of first arrival times up to a final time $ j T$ into simulation on intervals $[1,T]$, $[T, 2T]$,..., $[(j-1)T,jT]$ then both the overall copula in the one-shot case and the step by step copula in the multi-step case are the same copula and lead to the same final result.


A simple characterization of the self-chaining condition guaranteeing common periods lack of memory for the multivariate exponential arrival-times is the following.

\begin{proposition} Let $L(v_1,v_2):= \log C(e^{v_1},e^{v_2})$, i.e. the log-copula with exponentially expressed arguments, defined for $v_1,v_2$ such that $C(e^{v_1},e^{v_2})>0$. We have
\begin{equation}\label{eq:copulacondhomog} C(u_1^k,u_2^k) = C(u_1,u_2)^k \iff L(k v_1,k v_2) = k L(v_1, v_2)
\end{equation}
for positive real $k$ and all $u_1,u_2$ where $C$ is strictly positive.
\end{proposition}
This amounts to homogeneity of order $1$ in the positive real $k$ for the log copula with exponential arguments.
The proof is immediate.

It is easy to notice that, for known families of copulas, Property~(\ref{eq:copulacondhomog}) is satisfied by the Marshall-Olkin copula and the Gumbel-Hougaard copula.

\begin{proposition}
The Gumbel-Hougaard copula and the Marshall-Olkin copula are self-chaining.
\end{proposition}

\begin{problem}{(\bf Characterization of Self Chaininig Copulas)} 
Find a characterization of self-chaining copulas, including Marshall-Olkin and Gumbel-Hougaard as a special case. In a way, this amounts to characterize homogeneous functions $L$ (i.e. satisfying Eq \ref{eq:copulacondhomog}) whose exponential $\exp(L)$ satisfies Eqs. (\ref{cop:1}, \ref{cop:2}, \ref{cop:3}).

The homogeneity condition may also be expressed as a PDE. Starting from
\[ L(tx, ty) =t L(x,y) , \]
and taking the total derivative wrt $t$ on both sides one obtains
\[ L_x(tx, ty) x + L_y(tx,ty) y = L(x,y)\]
where $L_x(\bar{x},\bar{y}) = \partial_x L(\bar{x},\bar{y})$ and similarly for $L_y$.
Set $t=1$ to get
\begin{equation}\label{L:PDE} L_x(x, y) x + L_y(x,y) y = L(x,y) . \end{equation}
Remembering that $L(v_1,v_2):= \log C(e^{v_1},e^{v_2})$ one has
\[  C_u (u,v)\ u\  \log(u)+ C_v(u,v)\ v\ \log(v) = C(u,v)\ \log C(u,v) . \]
Now the problem is: Characterizing all copula functions, i.e. all positive functions $C$ with values in $[0,1]$ and satisfying Eqs (\ref{cop:1}, \ref{cop:2}, \ref{cop:3}), that satisfy this PDE. 

Alternatively, one can work with the $L$ PDE (\ref{L:PDE}), with $L$ taking values in $(-\infty,0]$, and such that $\exp(L)$ satisfies
Eqs (\ref{cop:1}, \ref{cop:2}, \ref{cop:3}).


\end{problem}

In the strict Archimedean copulas context, with copulas associated with frailty distributions, we can solve the above problem and provide a full characterization of the self chaining property.


Before solving the above problem in the archimedean context, it is worth noticing that the self-chaining Marshall-Olkin and Gumbel-Hougaard copulas are associated resprectively with the multivariate Marshall-Olkin and Gumbel (Eq. (\ref{eq:Gumbexp})) multivariate exponential distributions mentioned in Section \ref{sec:introMVLM} when one standardizes away the marginals. Such distributions  satisfy the CPLM property as we had already observed in the abovementioned Section.  

However, despite such consistency, the Marshall-Olkin and Gumbel-Hougaard copulas are far from being the most commonly used copula function. The most commonly used copula functions are the Gaussian copula and the Student-t copula. 

Gaussian Copula functions have been associated by part of the press to the financial crisis started in 2007. For a discussion on a number of misunderstandings associated with this opinion and on the real shortcomings of the way Gaussian copulas have been used for Collateralized Debt Obligations (CDOs) see \cite{brigopallatorre}. 

\begin{remark}{(\bf Iterating a Gaussian copula for multivariate arrival times sampling is inconsistent).}
It is worth saying immediately that the Gaussian copula is NOT self-chaining and therefore is not consistent with multivariate lack of memory. Iterating a Gaussian copula for joint arrival times simulation is thus wrong, and kills dependence in the long run, as we have seen in Example \ref{ex:iter} above.
\end{remark}
In the following we turn to characterization of self-chaining copulas in the Archimedean family.

\section{Characterizing Archimedean self-chaining copulas}\label{sec:archchar}
To accomplish the above characterization within Archimedean copulas, we need to narrow a little the focus to Archimedean copulas whose generator is associated to a Frailty distribution. The key result in this respect is Kimberling's (1974) characterization of Archimedean copulas with a generator working for every dimension:

\begin{theorem}\label{th:kimberling} {\bf (Kimberling, 1974)}
Given a strict generator $\varphi$,\\ the function $(u_1,\ldots,u_n) \mapsto \varphi^{-1} (\varphi(u_1)+..+\varphi(u_n))$ from $[0,1]^n$ to $[0,1]$ is a copula for every $n$ if and only if $\varphi^{-1}$ is completely monotone on $[0,\infty)$.
\end{theorem}
In turn, the Bernstein-Widden Theorem (Widder 1946)  states that complete monotonicity is equivalent to being the Laplace transform of a non-negative Random variable measure, that is called the frailty distribution. Therefore, given a strict generator $\varphi$, this generates a strict Archimedean copula for every dimension $n$ if and only if the inverse generator $\varphi^{-1}$ is the Laplace transform of a frailty distribution. A discussion on this characterization and some interesting cases where complete monotonicity does not hold are in McNeil and Neslehova (2007).

Before we proceed further we also need some definitions and useful propositions. We follow Sato (1999) and for a probability measure $\mu$ we define its Fourier transform with $\hat{\mu}$.

\begin{definition}
A probability measure $\mu$ is \emph{infinitely divisible} if for each $N\in\Natural$ there exists a probability measure $\mu_N$ such that
\[
\hat{\mu} = (\hat{\mu}_N)^N
\]
\end{definition}

\begin{theorem}
If an Archimedean copula is self chaining, then its frailty distribution is infinitely divisible.
\end{theorem}
\begin{proof}
For all $z\in[0,+\infty)$ set $u=\phi^{-1}(z/2)\in(0,1]$. Now write
\[
\phi^{-1}(z)^{1/N} = \phi^{-1}(2\phi(u))^{1/N} = \phi^{-1}(2\phi(u^N)) = \phi^{-1}(2\phi(\phi^{-1}(z/2)^{1/N}))
\]
Therefore, if we set $\phi^{-1}_N(z) = \phi^{-1}(2\phi(\phi^{-1}(z/2)^{1/N}))$ we can write
\[
\phi^{-1}(z) = \phi^{-1}_N(z)^N
\]
Since the inverse generator is the Laplace transform of the frailty, the same relationship will hold for the characteristic function. Therefore the frailty is infinitely divisible.
\end{proof}

\begin{definition}
Let $\mu$ be an infinitely divisible probability measure or $\Real$. It is called \emph{strictly stable} if, for all $a>0$, there is $b>0$ such that
\[
\hat{\mu}(z)^a = \hat{\mu}(bz)\text{.}
\]
It is called \emph{strictly semi-stable} if for some $a>0$ with $\alpha\neq 1$, there is $b>0$ satisfying the above relationship. The stability index is the value $\alpha$ that solves
\[
a\left| b \right|^{-\alpha} = 1
\]
\end{definition}

\begin{theorem}
If an Archimedean copula is self chaining, then its frailty distribution is strictly semi-stable.
\end{theorem}
\begin{proof}
Consider an Archimedean copula with generator $\varphi$, and define a set if indices $I$. Then, for a vector $u=\{u_i\}_{i\in I} \in [0,1]^I$ the copula will be given by
\[
C\left(\{u_i\}_{i\in I}\right) = \phi^{-1}\left( \sum_{i\in I} \phi(u_i) \right)
\]
For the copula to have the self-chaining property we demand that for all $\ell > 0$
\[
C\left(\{u^\ell_i\}_{i\in I}\right)= C\left(\{u_i\}_{i\in I}\right)^\ell
\]
For each $N\in\Natural$ we therefore compute the copula value at a point with $N$ abscissas equal to $u$ and the rest equal to one. Thus, the self-chaining property implies
\[
\phi^{-1}\left( N \phi(u^\ell) \right) = \phi^{-1}\left( N \phi(u) \right)^\ell
\]

We now define an auxiliary function $g$ and its inverse as
\begin{align*}
g &: [-\infty,0] \rightarrow [0,+\infty] : x \rightarrow \phi( \exp x )\\
g^{-1} &: [0,+\infty] \rightarrow [-\infty,0] : y \rightarrow \log \phi^{-1}(y)
\end{align*}
Taking logs of both sides of the self-chaining property allows us to write it in terms of $g$ in the following way
\[
g^{-1} \left( N g(\ell x) \right) = \ell g^{-1} \left( N g(x) \right)
\]
Thus, if we define a second auxiliary function $G_N(x) = g^{-1}(N g(x))$ the above relationship becomes
\[
G_N(\ell x) = \ell G_N(x) \quad \forall \ell > 0, x<0
\]
This implies that $G_N$ is a linear function without a constant. To see that, set $x=-1$, $z=-\ell$ and $k_N = G_N(-1)$. Then write
\[
G_N(z) = -z G_N(-1) = - k_N z \quad \forall z < 0
\]

Having established the form of $G_N$ we substitute backwards. writing $g^{-1}(N g(z)) = - k_N z$, and by the change of variable $x = Ng(z)$ we write that
\[
g^{-1}(x) = -k_N g^{-1}(x/N)
\]
As $g^{-1}(x) = \log \varphi^{-1}(x)$ we conclude that self-chaining implies a generator that satisfies the relationship
\[
\left[ \varphi^{-1}(x)\right]^{-1/k_N} = \varphi^{-1}(x/N)
\]

The inverse generator is the Laplace transform of the frailty distribution. The Fourier transform will be given of course by $\hat{\mu}(z) = \phi^{-1}(-\ii z)$. The relationship will then carry onwards to the characteristic function of the frailty.

We have therefore shown that there exist \emph{epochs} $a_N = -1/k_N$, for which we set \emph{spans} $b_N = 1/N$, satisfying the relationship
\[
F(x)^{a_N} = F(b_Nx)
\]
By definition we conclude that the frailty distribution is \emph{strictly semi-stable}. We denote the stability index with $\alpha$.
\end{proof}

From the definition of the stability index we can express the relationship between the spans and their corresponding epochs. In particular
\[
a_N\left|b_N\right|^{-\alpha} = 1 \Rightarrow k_N = - N^\alpha
\]

It is possible to characterize the frailty distribution even further, since the generator of the frailty distribution is continuous. To do so we also need to define the measure operator $T_r$ for each $r>0$, which transforms a measure $\mu$ on $\Real$ as
\[
(T_r \mu) (B) = \mu(r^{-1} B) \text{, for Borel }B\in\mathcal{B}(\Real)
\]

In addition we provide a useful theorem (Sato, Thm 14.3)
\begin{theorem}
Let $\mu$ be infinitely divisible with Levy measure $\nu$. Let $0<\alpha<2$.
\begin{enumerate}
\item $\mu$ is $\alpha$-semi-stable with $b>1$ as a span iff $\nu$ satisfies
\[
\nu = b^{-\alpha}T_b \nu
\]
\item $\mu$ is $\alpha$-stable iff the above relationship holds for all $b>0$.
\end{enumerate}
\end{theorem}

\begin{theorem}
If a strict Archimedean copula is self chaining then its frailty distribution is strictly stable.
\end{theorem}
\begin{proof}
We have shown that the frailty is strictly semi-stable, with spans $b_N = 1/N$ and associated epochs $a_N = -(1/N)^\alpha$, for all $N\in \Natural$. It is also straightforward to verify that the quantities $\tilde{b}_N=N$ are also valid spans, with associated epochs $\tilde{a}_N = -N^\alpha$.

Therefore, for any rational number $b = M/N$ and any Borel $B\in\mathcal{B}(\Real)$ we can write
\[
b^{-\alpha} \left( T_b \nu \right)(B) = \left( \frac{M}{N} \right)^{-\alpha} \left( T_{M/N} \nu\right)(B) = \left( \frac{M}{N} \right)^{-\alpha} \left( T_M T_{1/N} \nu \right)(B) = \nu(B)
\]

Since the distribution is strictly semi-stable for all spans $b=M/N$, $M,N\in\Natural$, by continuity it is strictly semi-stable for all spans $b>0$, and is therefore \emph{strictly stable}.
\end{proof}

Finally, we can characterize continuous self-chaining copulas completely.

\begin{theorem}
If a strict Archimedean copula is self chaining, then it is the Gumbel-Hougaard copula.
\end{theorem}
\begin{proof}
Having established that the frailty is strictly stable with positive support we can write the characteristic function in closed form. This will in turn determine the generator.

In general, stable distributions will have characteristic functions given by
\begin{align*}
L(-\ii z) &= \exp \left\{ -c|z|^\alpha \left( 1-\ii \beta \tan \frac{\pi\alpha}{2} \sgn z \right)  + \ii \gamma z\right\}\text{, if }\alpha\neq 1\\
L(-\ii z) &= \exp \left\{ -c|z| \left( 1+\ii \beta \frac{2}{\pi}\sgn z \log |z| \right)  + \ii \gamma z\right\}\text{, if }\alpha= 1
\end{align*}
for $c>0$, $\beta \in [-1,1]$ and $\gamma\in\Real$. These three parameters, together with the stability index $\alpha$ determine the density.

As the frailty has positive support the restrictions $\beta=1$ and $\gamma=0$ are imposed. Under these constraints the Laplace transform, which will also be the frailty generator, will be defined only for $\alpha \neq 1$, and is given by
\[
L(z) = \exp \left\{ -c|z|^\alpha \left( 1+ \tan \frac{\pi\alpha}{2} \right) \right\}\text{, for }\alpha\neq 1
\]

Also, as the Laplace transform is decaying the further constraint
\[
1+ \tan \frac{\pi\alpha}{2} > 0 \Rightarrow \alpha < 1
\]
In this case the quantity $\delta = c\left(1+ \tan \frac{\pi\alpha}{2}\right) > 0$ is a constant, and the inverse generator takes the Gumbel-Hougaard form.
\end{proof}

\section{Relationship with extreme value copulas}\label{sec:sccevc}
We found out later, after proving the above result, that such result is already known in a different context. In particular, it is a result associated with max stable copulas (MSC) and Extreme Value Coupulas. A MSC is a copula $C$ that is obtained as follows. Consider a multivariate distribution for a random vector
$[X_1,\ldots,X_n]$ having $C=C_X$ as copula, take $k$ iid copies of the vector given by 
\[ \left[X_{i,1},  X_{i,2} , \ldots , X_{i,n} \right], \ \ \   i=1,\ldots,k\]
and define
\[ M_{kj} = \max\left(\begin{array}{c}X_{1,j} \\X_{2,j}\\ \vdots \\ X_{k,j} \end{array}\right) \ \ for \   j \le  n .\] 
Consider the copula 
\[ C_M(u) = C_{M_{k,1},M_{k,2},\ldots, M_{k,n}}(u_1,\ldots,u_n) = C_{X_{1},X_{2},\ldots, X_{n}}(u_1^k,\ldots,u_n^k)^{1/k} \]
The copula $C_X$ is said to be Max-Stable if 
\[   C_X(u) = C_M(u)   \]
namely if 
\[  C(u_1^k,\ldots,u_n^k)^{1/k} =  C(u_1,\ldots,u_n)\]
which is precisely our definition of self chaining copula. 

One can see that the max stable copulas (or, equivalently, extreme value copulas, see Theorem 3.3.5 in \cite{Nel}) follow the same characterization of our self-chaining copulas above, and it is known that the only archimedean max-stable copula is the Gumbel-Hougaard copula, see \cite{genestrivest} or Theorem 4.5.2 in \cite{Nel}.

Besides the other case of the Marshall-Olkin copula, featuring a discrete component, we may fully characterize Extreme Value (and hence Self-Chaining) copulas by means of the Pickands function characterization.

Indeed, taking for simplicity again the bivariate case, basically our earlier homogeneity characterization translates in the two-variables homogeneous function being expressed as a one-variable function of the ratio. Such function consitutes the Pickands function of the copula, and Extreme Value Copulas (and hence Self-Chaining Copulas) can be fully characterized in terms of such function.  For the multivariate case, the Pickands function is the restriction of the $n$-variables homogeneous function to the unit simplex.  For the details and the Pickands function properties see for example \cite{gudendorf} or \cite{mai}.  

\section{Conclusions, practical applications and further research}\label{sec:conclu}
This work investigates dependence across marginally exponentialy distributed inter-arrival times, for example default times in financial modeling or inter-failure times in reliability theory. 

In many applications one needs to simulate the random times progressively, checking "survival" in subsequent intervals 
\[ [0, T), \  [T, 2T),\   \  [2T, 3T), \ldots, \  [(N-1)T, N T]    \]
rather than at the final time $N T$ directly.  If one assumes exponential marginals for the random times and a given dependence structure (copula) in checking whether the arrival times "survive" each subinterval\\ $[(i-1) T, i T)$ conditional on having survived  $(i-1)T$, $i=1,\ldots,N$, can one deduce the properties the dependence structure needs to feature to be consistent in the interval $[0, N T]$ as a whole?
As we have shown above, this condition is not easily satisfied. For example, iterating a Gaussian copula is not consistent with a Gaussian copula for  the total $[0, \ NT]$ interval, and in fact the inconsistent dependence obtained via iteration will be weaker than the dependence implied by the same Gaussian copula applied directly and consistently to the single $[0, \ NT]$, as we have seen in Example \ref{ex:iter} above. 

The self-chaining condition we introduce in this paper is necessary and sufficient to guarantee that a multivariate lack of memory property holds consistently with a given dependence structure. The dependence structure  is consistent for the iterated and terminal sampling. This means that we may choose the copula in the self-chaining family  and then later decide whether we wish to sample times occurrences in a step by step fashion or in a one-shot fashion {\em with the same copula} without introducing inconsistencies.

We point out that the self-chaining condition is the same as the extreme value copula condition and that, in the context of the Archimedean copulas associated with frailty distributions, the self-chaining condition forces the copula to be the Gumbel-Hougaard copula, thus providing a full characterization. The other known explicit case satisfying the self-chaining condition is outside the Archimedean family and is given by the Marshall-Olkin copula, featuring a discrete component and a positive probability that $\tau_1 = \tau_2$. More generally, self-chaining copulas are characterized via a homogeneity condition for the log-copula, or a related partial differential equation, which translates into the Pickands function characterisation.

Further research should focus on multivariate self-chaining copulas of dimension larger than 2 and with a flexible and rich parametric structure. It would be ideal if the parametric dependence were expressed via a matrix. Two such matrix-parametrized  copulas are the Gaussian and Student-t copulas which, however are not self-chaining.   
A possible research idea is trying to extend the bivariate self-chaining limit of the t copula (see for example \cite{gudendorf}) to the multivariate case.

\end{document}